\documentclass[a4paper,10pt,twoside]{article}
\usepackage[utf8]{inputenc}
\usepackage{amsmath,amsthm,amsfonts,amssymb,bbm}
\usepackage{graphicx,psfrag,color}%,subfigure
\usepackage{cite}
\usepackage[colorlinks=true]{hyperref}
\usepackage[utf8]{inputenc}
\usepackage[a4paper,portrait,left=3.5cm,right=3.5cm,top=3.5cm,bottom=3.5cm]{geometry}

\usepackage{fancyhdr}

\newtheorem{theorem}{Theorem}
\newtheorem{proposition}{Proposition}
\newtheorem{remark}{Remark}
\newtheorem{lemma}{Lemma}

\newcommand{\Or}{\mathcal{O}}
\newcommand{\Pb}{\mathbbm{P}}
\newcommand{\Id}{\mathbbm{1}}
\newcommand{\e}{\varepsilon}
\newcommand{\I}{\mathrm{i}}
\newcommand{\D}{\mathrm{d}}
\newcommand{\C}{\mathbb{C}}
\newcommand{\R}{\mathbb{R}}
\newcommand{\N}{\mathbb{N}}
\newcommand{\Z}{\mathbb{Z}}
\renewcommand{\Re}{\mathrm{Re}}

\renewcommand{\phi}{\varphi}
\renewcommand{\textrm}{\textnormal}

\begin{document}
\title{Upper tail large deviations for Brownian motions with one-sided collisions} 
\author{Thomas~Weiss\footnote{Acoustics Research Institute, Austrian Academy of Sciences, Vienna. {thomas.weiss@oeaw.ac.at}.}
    }
%\KEYWORDS{Reflected Brownian motion; KPZ universality class; large deviations}
\maketitle
\begin{abstract}
    The system of interacting Brownian motions, where a particle is reflected asymmetrically from its left neighbor, belongs to the KPZ universality class, with multi-point asymptotics having been derived in previous works. In this paper we show upper tail large deviation principles for all three fundamental initial conditions, including explicit calculation of the rate function. For the periodic case the Lambert-W function, which is already present in the Fredholm determinant formula, also appears in the rate function.
\end{abstract}

\pagestyle{fancy}
\fancyhead{}
\fancyhead[CE,CO]{Large deviations for Brownian motions with one-sided collisions}

\section{Introduction}

At the beginning of the growing interest in the famous KPZ equation around the millenium, explicit solutions or limit results of this equation were still out of reach. 
The focus of the scientific community thus shifted onto integrable models in the \emph{KPZ universality class}, which are models that share the characteristic scaling exponents and limiting processes with the KPZ equation, with the TASEP \cite{Jo00b} and the PNG model being two of the first models where limit results were proven \cite{PS02}.

After the TASEP, the model of reflected Brownian motions studied here is the one where asymptotics have been worked out in the greatest comprehensiveness. 
Multi-point scaling limits are established for all six fundamental KPZ initial profiles \cite{WFS17}, and it was also the first model where the two-time distribution has been established \cite{Joh17}.

This model consists of a set Brownian motions $x_n(t)$ with $n\in\N$ or $n\in\Z$, where the particle $x_n$ obeys a one-sided reflection at the particle $x_{n-1}$ in the positive direction. 
The ordering relation $x_n\geq x_{n-1}$ has to be satisfied by the initial condition and persists for all $t>0$ by the reflection.

The system $\{x^\textrm{packed}_n(t),n\in\N\}$ is subject to the initial condition $x_n^\textrm{packed}(0)=0$, and is the simplest one to study. 
Well-definedness is evident e.g. using the Skorokhod construction of the reflection, recognizing that $x_n$ is a deterministic function of only finitely many driving Brownian motions.

The second system $\{x_n^\textrm{flat}(t),n\in\Z\}$ is given by the periodic initial condition $x^\textrm{flat}_n(0)=n$. 
It is computationally more involved and, quite surprisingly, the Lambert W function appears in the contour integral formula for the kernel~\cite{FSW13}.

Finally, depending on a parameter $\rho>0$, we define the system $\{x_n^{\textrm{stat},\rho}(t),n\in\Z\}$ satisfying the random initial condition $x_0^{\textrm{stat},\rho}(0)=0$ and $x_n^{\textrm{stat},\rho}(0) - x_{n-1}^{\textrm{stat},\rho}(0) = \zeta_n$, where $\zeta_n$ are independent  random variables distributed as $\zeta_n\sim\mathrm{Exp}(1)$ for $n\geq1$ and  $\zeta_n\sim\mathrm{Exp}(\rho)$ for $n\leq0$. The system $\{x_n^\textrm{stat}(t)\}:=\{x_n^{\textrm{stat},1}(t)\}$ is stationary in space and time \cite[Proposition~3.4]{FSW15}.

For the last two systems, well-definedness is not a priori clear as for the packed case. In fact, systems indexed by $\Z$ need the requirement of the initial conditions being spaced not too close together \cite[Section~2.3]{WFS17}.

This paper derives large deviations principles with explicit rate functions for all three fundamental initial conditions of the system of Brownian motions with one-sided reflections.

The packed initial conditions case is in fact equivalent to the GUE random matrix ensemble and has been first proven in \cite{MS14}. 
For the stationary case, Lyapunov exponents but no rate function were derived using the equivalent definition of Brownian directed percolation \cite{Jan19}.
The O'Connell-Yor polymer, from which the model discussed here can be recovered by setting the temperature parameter to zero, has also been studied for large deviations \cite{Jan15}.

More results are available for the TASEP, from which our system is a low-density limit \cite{KPS12}: 
The upper tail rate function was first derived as a indirect variational formula \cite{Jo00b}. 
By now, arbitrary initial conditions and even the much more involved multi-point large deviations have been worked out in the outstanding contribution \cite{QT21}. 
Recently also the upper tail rate function of the related ASEP was proven \cite{DZ22}.

\section{Main Results}
For the purpose of stating our main results, let us say that a system $\{x_n(t),n\in\Z\}$ satisfies an upper tail large deviation principle with rate function $r(a)$ if and only if for any $a>0$ the following limit exists and holds:
\begin{equation}
  \lim_{t\to\infty}t^{-1}\log\Pb\left(x_t(t)\geq2t+at\right)=-r(a).
\end{equation}

\begin{theorem}\label{thmStepRate}
The system $\{x^\textrm{packed}_n(t),n\in\N\}$ satisfies an upper tail large deviation principle with rate function
\begin{equation}\label{eqStRate}
 r^\mathrm{packed}(a)=(2+a)\sqrt{a+\frac{a^2}{4}}+2\log\left(1+\frac{a}{2}-\sqrt{a+\frac{a^2}{4}}\right).
\end{equation}
\end{theorem}
This is a known result \cite{MS14,Jan19} which is nevertheless proven here, as most parts of the proof are prerequisites for the other two cases, especially the stationary case.

\begin{theorem}\label{thmFlatRate}
The system $\{x_n^\textrm{flat}(t),n\in\Z\}$ satisfies an upper tail large deviation principle with rate function
\begin{equation}\label{eqFlRate}
 r^\mathrm{flat}(a)=\max_{z_a\in(-\infty,-1)}\left\{\Big(\phi(z_a)-z_a\Big)\left(\frac{z_a+\phi(z_a)}{2}+1+a\right)\right\},
\end{equation}
where the maximum is attained by the unique solution to 
\begin{equation}\label{eqza}
    (z_a+1)(\phi(z_a)+1)+a=0.
\end{equation}
\end{theorem}

\begin{proposition}\label{propFlatAsy}
 The function $r^\mathrm{flat}(a)$ shows the asymptotic behavior
 \begin{equation}\label{eqFlRateAsy}\begin{aligned}
   r^\mathrm{flat}(a)&=\frac{4}{3}a^{3/2}+\Or\Big(a^2\Big) & &\text{as } a\to0,\\
   r^\mathrm{flat}(a)&=\frac{(a+1)^2}{2}+\Or\Big(a^3e^{-a}\Big) & &\text{as } a\to\infty.
\end{aligned}
 \end{equation}
\end{proposition}
As expected, the small $a$ limit matches the right tail of the GOE Tracy-Widom distribution. For large $a$ the particle $x_t$ is unusually far towards the right, so it does not significantly collide with its left neighbour anymore. It is thus reasonable that the rate function approaches the one of a free Brownian motion. However, the exponentially fast convergence is quite interesting.

\begin{figure}
\centering
\includegraphics[height=7cm]{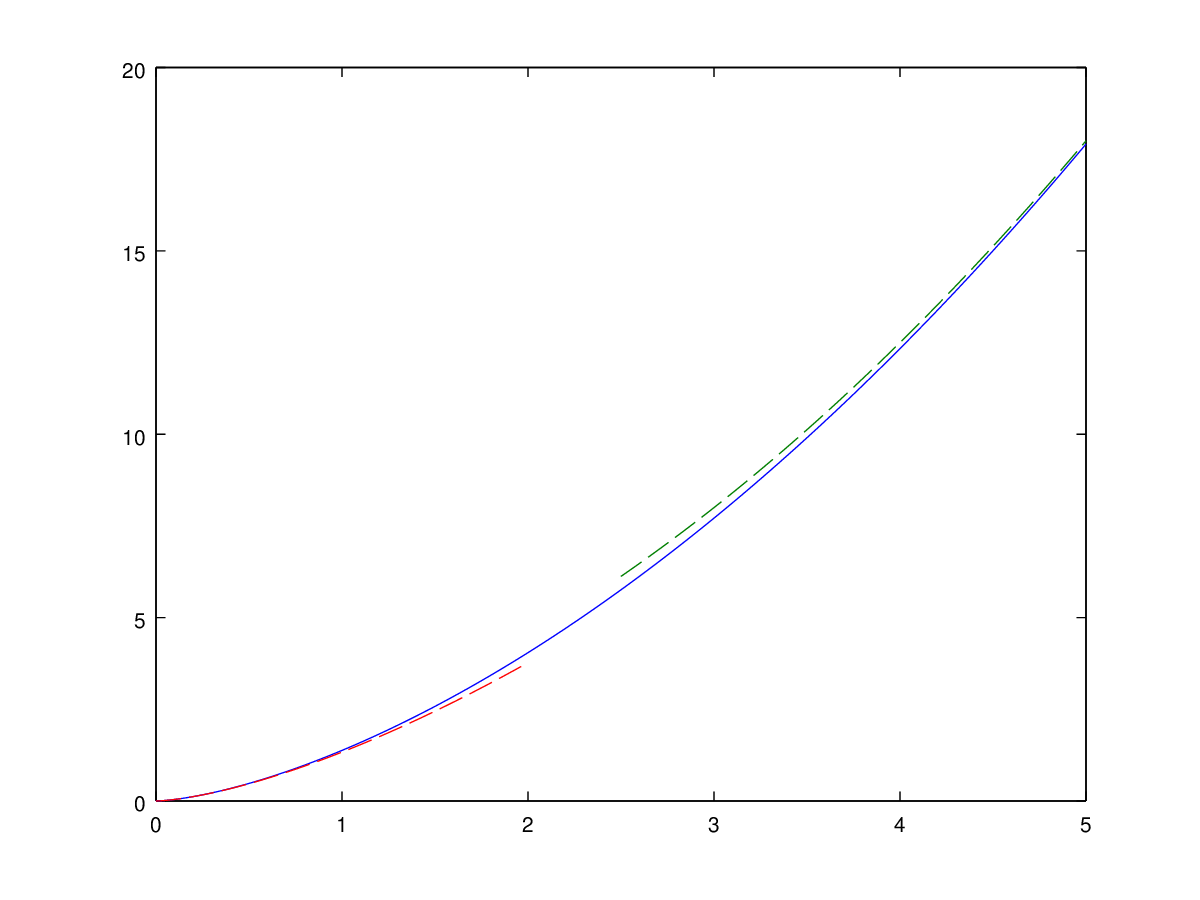}
\caption{The rate function $r^\mathrm{flat}(a)$ (solid line) and its approximating functions according to \eqref{eqFlRateAsy} at $a\to0$ and $a\to\infty$ (dashed lines).}%\vspace*{-3pt}
\label{figFlRate}
\end{figure}

\begin{theorem}\label{thmStatRate}
The system $\{x_n^\textrm{stat}(t),n\in\Z\}$ satisfies an upper tail large deviation principle with rate function
\begin{equation}%\label{eqFlRate}
 r^\mathrm{stat}(a)=-\frac{a^2}{4}+\left(1+\frac{a}{2}\right)\sqrt{a+\frac{a^2}{4}}+\log\left(1+\frac{a}{2}-\sqrt{a+\frac{a^2}{4}}\right).
\end{equation}
\end{theorem}

\begin{remark}
  The function $r^\mathrm{stat}(a)$ shows the asymptotic behavior
 \begin{equation}\label{eqStRateAsy}\begin{aligned}
   r^\mathrm{stat}(a)&=\frac{2}{3}a^{3/2}+\Or\left(a^2\right) & &\text{as } a\to0,\\
   r^\mathrm{stat}(a)&=a+\frac{1}{2}-\log a+\Or\left(a^{-1}\right) & &\text{as } a\to\infty.
\end{aligned}
 \end{equation}

\end{remark}

\section{Packed initial condition}\label{secPac}
The starting point of our analysis is the determinantal formula for the multi-point distribution, given by \cite[Proposition~5.1]{WFS17}, which is reduced to the one-point distribution in the following corollary:
\begin{proposition}\label{proSteppKernel}
For any $n\geq0$, it holds
\begin{equation}\label{eqStep33}
\Pb\bigg(x^\textrm{packed}_n(t)\leq s\bigg)=\det(\Id-P_s K_{n,t}^\textrm{packed} P_s)_{L^2(\R)},
\end{equation}
where $P_s(\xi)=\Id_{\xi>s}$ and the kernel $K_{n,t}^\textrm{packed}$ is given by
\begin{equation}\label{eqKpacked}
    K_{n,t}^\mathrm{packed}(\xi_1,\xi_2)=\frac{1}{(2\pi\I)^2}\int_{\I\R-\e}\D w\,\oint_{\Gamma_0}\D z\,\frac{e^{{t} w^2/2+\xi_1w}}{e^{{t} z^2/2+\xi_2 z}}\frac{(-w)^{n}}{(-z)^{n}}\frac{1}{w-z}.
\end{equation}
\end{proposition}

\begin{proof}[Proof of Theorem \ref{thmStepRate}]
 We start by defining a modified kernel, which is the result of applying the scaling $n=t$, $s=2t+at$ as well as a conjugation, which does not change the determinant.
 	\begin{equation}\label{eq5}\begin{aligned}	\widehat{K}_{t}^\mathrm{packed}(\xi_1,\xi_2)&=e^{\xi_1-\xi_2}K_{n,t}^\mathrm{packed}(2t+at+\xi_1,2t+at+\xi_2)\\
 	&=\frac{1}{(2\pi\I)^2}\int_{\I\R-\e}\D w\,\oint_{\Gamma_0}\D z\,e^{tH(w)-tH(z)}\frac{e^{\xi_1(w+1)-\xi_2(z+1)}}{w-z},
\end{aligned}\end{equation}
where
\begin{equation}
 H(w)=\frac{w^2-1}{2}+(2+a)(w+1)+\log(-w).
\end{equation}
The probability of large deviations towards the right tail can then be calculated by
\begin{equation}\label{eq11}\begin{aligned}
 \Pb\left(x^\textrm{packed}_t(t)\geq 2t+at\right)&=1-\det(\Id-P_{2t+at} K_{t,t}^\textrm{packed} P_{2t+at})_{L^2(\R)}\\
 &=1-\det(\Id-P_0 \widehat{K}_t^\textrm{packed} P_0)_{L^2(\R)}\\
 &=\sum_{k=1}^\infty\frac{(-1)^{k+1}}{k!} \int_{\R_+^k}\D \xi_1\dots\D \xi_k\det_{1\leq i,j\leq k}[\widehat{K}_t^\textrm{packed}(\xi_i,\xi_j)],
\end{aligned}\end{equation}
which leaves us with the main open task of obtaining the asymptotics of the kernel $\widehat{K}_{t}^\mathrm{packed}$ as $t\to\infty$. The main tool to achieve this is \emph{steepest descent analysis}.

The function $H$ has two saddle points, which are found by setting the derivative  $H'(w)=w+2+a+1/w=0$, giving
\begin{equation}
 w_\pm = -1-\frac{a}{2}\pm\sqrt{a+\frac{a^2}{4}},
\end{equation}
where $w_+>-1+\sqrt{a/(a+2)}$ is a local maximum and $w_-<-1-a$ is a local minimum on the real line.

We specify the contours of the integrals in \eqref{eq5} to $w\in\gamma_-=\left\{w_-+it, t\in\R\right\}$ and $z\in\gamma_+=\left\{w_+e^{it}, t\in[0,2\pi)\right\}$. The contours satisfy the requirements for steepest descent analysis: The function $H$ evaluated at the contour $\gamma_-$ achieves its global maximum at the non-degenerate saddle point $w_-$ (ensuring stationary phase), while being strictly bounded away from the maximum value outside of a neighborhood of $w_-$ (ensuring exponentially fast decay there). Similarly, the function $-H$ along $\gamma_+$ has the critical point $w_+$ as its global maximum. Then all asymptotically relevant contributions come from a neighborhood of the critical points.
\begin{equation}\begin{aligned}\label{eq9}
  \lim_{t\to\infty}&te^{tr^\mathrm{packed}(a)}\widehat{K}_{t}^\mathrm{packed}(\xi_1,\xi_2)=\lim_{t\to\infty}te^{-t(H(w_-)-H(w_+))}\widehat{K}_{t}^\mathrm{packed}(\xi_1,\xi_2)\\
  &=\frac{2\pi}{\sqrt{-H''(w_+)H''(w_-)}}\frac{e^{\xi_1(w_-+1)-\xi_2(w_++1)}}{w_--w_+}
  =:\widehat{K}_\infty^\mathrm{packed}(\xi_1,\xi_2).
\end{aligned}\end{equation}

By this point-wise limit, the $k$-th term of the sum in \eqref{eq11} should decay exponentially as $e^{-ktr^\mathrm{packed}(a)}$, so that the only asymptotically relevant part is the leading order term $k=1$.

However, for convergence and interchangeability of the limit with the series and integral, we still need a uniformly integrable/summable bound. We obtain the kernel bound by noticing $\Re(w+1)=w_-+1<-a$ on $\gamma_-$ and $\Re(z+1)\geq w_++1>\sqrt{a/(a+2)}$ on $\gamma_+$.
\begin{equation}\begin{aligned}	\left|\widehat{K}_{t}^\mathrm{packed}(\xi_1,\xi_2)\right|&\leq\frac{1}{(2\pi)^2}\int_{\gamma_-}|\D w|\,\oint_{\gamma_+}|\D z|\,\left|e^{tH(w)-tH(z)}\frac{e^{\xi_1(w+1)-\xi_2(z+1)}}{w-z}\right|\\
&=\frac{1}{(2\pi)^2}\int_{\gamma_-}|\D w|\,\oint_{\gamma_+}|\D z|\,\left|\frac{e^{tH(w)-tH(z)}}{w-z}\right|e^{\xi_1\Re(w+1)-\xi_2\Re(z+1)}\\
&\leq\frac{e^{-a\xi_1-\sqrt{a/(a+2)}\xi_2}}{(2\pi)^2}\int_{\gamma_-}|\D w|\,\oint_{\gamma_+}|\D z|\,\left|\frac{e^{tH(w)-tH(z)}}{w-z}\right|.
\end{aligned}\end{equation}
We can apply the same steepest descent argument as before to the integral part, which implies its exponential decay as $t\to\infty$. The following bounds (assuming $t\geq1$, and with $C_1$ depending on $a$ only) of various detail are an immediate consequence:
\begin{equation}\label{eq12}\begin{aligned}\left|\widehat{K}_{t}^\mathrm{packed}(\xi_1,\xi_2)\right|&<C_1e^{-a\xi_1-\sqrt{a/(a+2)}\xi_2}t^{-1}e^{-tr^\mathrm{packed}(a)}
\\&\leq C_1e^{-a\xi_1}e^{-tr^\mathrm{packed}(a)}\leq C_1e^{-a\xi_1}.\end{aligned}\end{equation}
Combining the kernel bound with the Hadamard inequality yields
\begin{equation}\label{eqHad}
    \begin{aligned}
        |\eqref{eq11}|&\leq\sum_{k=1}^\infty\frac{1}{k!} \int_{\R_+^k}\D \xi_1\dots\D \xi_k\left|\det_{1\leq i,j\leq k}[\widehat{K}_t^\textrm{packed}(\xi_i,\xi_j)]\right|\\
        &\leq\sum_{k=1}^\infty\frac{1}{k!} \int_{\R_+^k}\D \xi_1\dots\D \xi_k\, k^{k/2}\prod_{i=1}^kC_1e^{-a\xi_i}=\sum_{k=1}^\infty\frac{k^{k/2}}{k!}\frac{C_1^k}{a^k} < \infty,
    \end{aligned}
\end{equation}
where the sum converges by Stirling's approximation. Applying dominated convergence and the point-wise kernel limit, we arrive at
\begin{equation}\begin{aligned}
 &\lim_{t\to\infty}te^{tr^\mathrm{packed}(a)}\Pb\left(x_t(t)\geq 2t+at\right)\\
 &=\lim_{t\to\infty}\sum_{k=1}^\infty\frac{\big(-te^{tr^\mathrm{packed}(a)}\big)^{1-k}}{k!} \int_{\R_+^k}\D \xi_1\dots\D \xi_k\det_{1\leq i,j\leq k}[te^{tr^\mathrm{packed}(a)}\widehat{K}_t^\textrm{packed}(\xi_i,\xi_j)]
 \\&=\sum_{k=1}^\infty\frac{\lim_{t\to\infty}\big(-te^{tr^\mathrm{packed}(a)}\big)^{1-k}}{k!} \int_{\R_+^k}\D \xi_1\dots\D \xi_k\det_{1\leq i,j\leq k}[\widehat{K}_\infty^\textrm{packed}(\xi_i,\xi_j)]
 \\&=\int_{\R_+}\D \xi\,\widehat{K}_\infty^\textrm{packed}(\xi,\xi) = \frac{2\pi w_-w_+}{\left(w_--w_+\right)^2\sqrt{\left(w_-^2-1\right)\left(1-w_+^2\right)}} > 0.
\end{aligned}\end{equation}
This directly implies Theorem~\ref{thmStepRate}.
\end{proof}

\section{Periodic initial condition}
Again, we start from a determinantal formula for the distribution of the $n$-th particle, which is now a special case of \cite[Proposition 2.2]{FSW13}:
\begin{proposition}\label{propFlat}
For any $n\geq0$, it holds
\begin{equation}
\Pb\left(x^\textrm{flat}_n(t)\leq s\right)=\det(\Id-P_s K_{t,n}^\textrm{flat} P_s)_{L^2(\R)},
\end{equation}
where $P_s(\xi)=\Id_{\xi>s}$ and the kernel $K_{t,n}^\textrm{flat}$ is given by
\begin{equation}\label{eqKtflat}
\begin{aligned}
K_{t,n}^\textrm{flat}&(\xi_1,\xi_2)=\frac{1}{2\pi\I} \int_{\Gamma_-} \D z\,\frac{e^{t z^2/2} e^{z \xi_1}(-z)^{n}}{e^{t \varphi(z)^2/2} e^{\varphi(z) \xi_2} (-\varphi(z))^{n}}.
\end{aligned}
\end{equation}
The function $\varphi$ is defined by $\varphi(z)=L_0(z e^{z})$, with $L_0$ denoting the Lambert-W function, \textit{i.e.}, the principal solution for $w$ in $z=w e^w$, and $\Gamma_-$ is any path going from $\infty e^{-\theta \I}$ to $\infty e^{\theta \I}$ with $\theta\in [\pi/2,3\pi/4)$, crossing the real axis to the left of $-1$, and such that $\varphi(z)$ is continuous and bounded.
\end{proposition}

\begin{proof}[Proof of Theorem \ref{thmFlatRate}]
 With $z/\phi(z)=e^{-z+\phi(z)}$ the kernel transforms to
 \begin{equation}
  K_{t,n}^\textrm{flat}(\xi_1,\xi_2)=\frac{1}{2\pi\I} \int_{\Gamma_-} \D z\,\frac{e^{t z^2/2} e^{z (\xi_1-n)}}{e^{t \varphi(z)^2/2} e^{\varphi(z) (\xi_2-n)}}.
 \end{equation}
Applying a conjugating factor and inserting the scaling, we arrive at a new kernel
 \begin{equation}\begin{aligned}\label{eq8}
  \widehat{K}_t^\textrm{flat}(\xi_1,\xi_2)&=e^{\xi_1-\xi_2}K_{t,t}^\textrm{flat}(2t+at+\xi_1,2t+at+\xi_2)\\&=\frac{1}{2\pi\I} \int_{\Gamma_-} \D z\, e^{t G(z)}e^{\xi_1(z+1)-\xi_2(\phi(z)+1)},
 \end{aligned}\end{equation}
where
\begin{equation}\label{eqG}
 G(z)=\frac{z^2-\phi(z)^2}{2}+(1+a)(z-\phi(z))=\frac{(z+1)^2}{2}-\frac{(\phi(z)+1)^2}{2}+a(z-\phi(z)).
\end{equation}
We find the critical points by setting $G'(z)=0$. Using $\phi'(z)z(1+\phi(z))=(1+z)\phi(z)$, we can simplify
\begin{equation}
 \begin{aligned}
  G'(z)&=z+1-(\phi(z)+1)\phi'(z)+a(1-\phi'(z))=\\
  &=z+1-\frac{(z+1)\phi(z)}{z}+a\cdot\frac{z(1+\phi(z))-(1+z)\phi(z)}{z(\phi(z)+1)}\\
  &=\frac{z-\phi(z)}{z(\phi(z)+1)}\Big((z+1)(\phi(z)+1)+a\Big).
 \end{aligned}
\end{equation}
As $z$ and $\phi(z)$ occupy disjoint branches of the Lambert function, the derivative vanishes exactly if $z$ satisfies \eqref{eqza}.

We have to show that \eqref{eqza} has indeed a unique solution $z_a\in (-\infty,-1)$ for each $a>0$. Recognizing that the function $z\mapsto ze^z$, $(-1,0)\rightarrow (-e^{-1},0)$ is bijective and increasing, its inverse shares these properties. Together with $z\mapsto ze^z$, $(-\infty,-1)\rightarrow (-e^{-1},0)$ being bijective and decreasing, we see that $\phi$ is bijective and decreasing when considered on $(-\infty,-1)\rightarrow (-1,0)$. Consequently, the function $z\mapsto (z+1)(\phi(z)+1)$, $(-\infty,-1)\rightarrow (-\infty,0)$ is bijective and increasing, implying the existence and uniqueness of the claimed $z_a$.

\begin{figure}
\centering
\includegraphics[height=5cm]{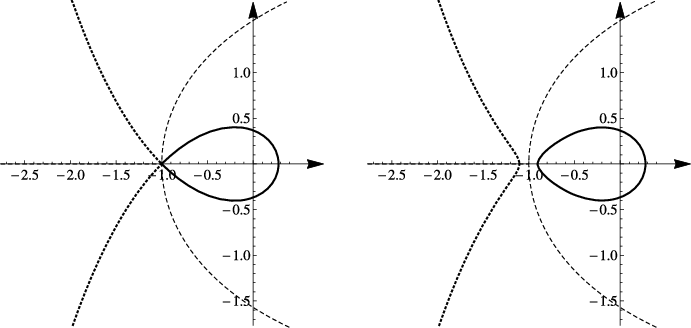}
\caption{The contour $\gamma$ (dotted line) and its image under $\phi$ (solid line)
 for the limiting case $a=0$ (left picture) and 
some positive $a$ (right
picture). The dashed lines separate the ranges
of the principal branch $0$ (right) and the branches $1$ (top left)
and $-1$ (bottom left) of the Lambert W function.}%\vspace*{-3pt}
\label{figGamma}
\end{figure}

In \eqref{eq8} we specify the contour $\Gamma_-$ to $\gamma$, defined by
\begin{equation}
 \gamma(\tau
)=L_{\lfloor\tau\rfloor+1_{\tau\geq0}}\left(z_ae^{z_a+2\pi\I\tau}\right),\tau\in\R.
\end{equation}
This contour passes through the critical point $z_a$ and is a steep descent curve by Lemma~\ref{lemContour}. It satisfies the differential identity $\gamma'(\tau)=2\pi\I\gamma(\tau)/(1+\gamma(\tau))$, giving
 \begin{equation}
  \begin{aligned}
  \eqref{eq8}=\int_\R\D \tau\, e^{tG(\gamma(\tau))}\frac{\gamma(\tau)}{1+\gamma(\tau)}e^{\xi_1(\gamma(\tau)+1)-\xi_2(\phi(\gamma(\tau))+1)}.
  \end{aligned}
 \end{equation}
All asymptotically relevant contributions to this integral come from a neighbourhood of $\tau=0$, as this is the global maximum of $G(\gamma(\tau))$. Precisely,
\begin{equation}\label{eqLaplace}
 \lim_{t\to\infty}\sqrt{t}e^{-t G(z_a)}\widehat{K}_t^\textrm{flat}(\xi_1,\xi_2)=\sqrt{\frac{2\pi}{-\eta}}\frac{z_a}{1+z_a}e^{\xi_1(z_a+1)-\xi_2(\phi(z_a)+1)}=: \widehat{K}_\infty^\textrm{flat}(\xi_1,\xi_2),
\end{equation}
where 
\begin{equation}%\note{nachrechnen!}
 \eta=\frac{\D^2}{\D\tau^2}G(\gamma(\tau))\big|_{\tau=0}=4\pi^2(\phi(z_a)-z_a)\left(\frac{\phi(z_a)}{(\phi(z_a)+1)^2}+\frac{z_a}{(z_a+1)^2}\right).
\end{equation}
By $z_a<-1$ and $-1<\phi(z_a)<0$, $\eta$ is strictly negative. We have thus established the pointwise kernel limit implicating exponential decay of $\widehat{K}_t^\textrm{flat}$ with rate $t$ and rate function $G(z_a) = -r^\mathrm{flat}(a)$.

We are left with the task of finding uniformly integrable/summable bounds. By Lemma A.1 (3) and (9) \cite{FSW13}, the inequalities $z\leq z_a<-1$ and $\phi(z)\geq\phi(z_a)>-1$ hold for $z\in\gamma$, so we can find a upper estimate of the kernel similarly as in the previous section.

 \begin{equation}\begin{aligned}
  \left|\widehat{K}_t^\textrm{flat}(\xi_1,\xi_2)\right|&\leq\frac{1}{2\pi} \int_{\gamma} |\D z| \left|e^{t G(z)}\right|e^{\xi_1\Re(z+1)-\xi_2\Re(\phi(z)+1)}\\
  &\leq\frac{e^{\xi_1(z_a+1)-\xi_2(\phi(z_a)+1)}}{2\pi} \int_{\Gamma_-} |\D z| \left|e^{t G(z)}\right|\\
  &\leq C_2e^{(z_a+1)\xi_1-(\phi(z_a)+1)\xi_2}\leq C_2e^{(z_a+1)\xi_1}
 \end{aligned}\end{equation}
Kernel bound and Hadamard's inequality ensure uniform absolute convergence as in \eqref{eqHad}, finishing the proof by
\begin{equation}\begin{aligned}
 \lim_{t\to\infty}&\sqrt{t}e^{tr^\mathrm{flat}(a)}\Pb\left(x^\textrm{flat}_t(t)\geq 2t+at\right)
 =\lim_{t\to\infty}\sqrt{t}e^{tr^\mathrm{flat}(a)}\left(1-\det(\Id-P_0 \widehat{K}_t^\textrm{flat} P_0)_{L^2(\R)}\right)\\
 &=\lim_{t\to\infty}\sum_{k=1}^\infty\frac{\big(-\sqrt{t}e^{tr^\mathrm{flat}(a)}\big)^{1-k}}{k!} \int_{\R_+^k}\D \xi_1\dots\D \xi_k\det_{1\leq i,j\leq k}[\sqrt{t}e^{tr^\mathrm{flat}(a)}\widehat{K}_t^\textrm{flat}(\xi_i,\xi_j)]
 \\&=\sum_{k=1}^\infty\frac{\lim_{t\to\infty}\big(-\sqrt{t}e^{tr^\mathrm{flat}(a)}\big)^{1-k}}{k!} \int_{\R_+^k}\D \xi_1\dots\D \xi_k\det_{1\leq i,j\leq k}[\widehat{K}_\infty^\textrm{flat}(\xi_i,\xi_j)]
 \\&=\int_{\R_+}\D \xi\,\widehat{K}_\infty^\textrm{flat}(\xi,\xi) = \sqrt{\frac{2\pi}{-\eta}}\frac{z_a}{\left(1+z_a\right)\left(\phi(z_a)-z_a\right)} > 0.
\end{aligned}\end{equation}
\end{proof}

\begin{lemma}\label{lemContour}
 The contour $\gamma\colon \R\to\C$ being defined by
 \begin{equation}
 \gamma(\tau
)=L_{\lfloor\tau\rfloor+1_{\tau\geq0}}\left(z_ae^{z_a+2\pi\I\tau}\right),
\end{equation}
is continuous in $\tau$. The global maximum of $\Re\,G(z)$, where $G(z)$ is given by \eqref{eqG}, along $\gamma$ is at the point $z_a=\gamma(0)$. Furthermore, for every small $\delta>0$ there exists a $\e>0$ such that $\Re\,G(\gamma(\tau))<\Re\,G(z_a)-\e$ for all $\tau\in\R\setminus(-\delta,\delta)$.
\end{lemma}
\begin{proof}
 Equivalent to the proofs of \cite[Lemma A.1, Lemma A.2]{FSW13} with minor, straightforward modifications.
\end{proof}

\begin{proof}[Proof of Proposition \ref{propFlatAsy}]
 Around the branching point $-1$, $\phi$ follows the Taylor expansion
 \begin{equation}
  \phi(-1+h)=-1-h-\frac{2}{3}h^2+\Or(h^3).
 \end{equation}
Inserting this into \eqref{eqza}, we see that $a$ is of order $h^2$. To obtain the behaviour of $z_a$ for small $a$, we therefore use the ansatz $z_a=-1+\alpha \sqrt{a}+\beta a$, resulting in
\begin{equation}\begin{aligned}
 0&=(\alpha \sqrt{a}+\beta a)\left(-\alpha \sqrt{a}-\beta a -\frac{2}{3}\alpha^2 a+\Or(a^{3/2})\right)+a\\
 0&=(1-\alpha^2) a + \left(-2\alpha\beta-\frac{2}{3}\alpha^3\right)a^{2/3}+\Or(a^2).
\end{aligned}\end{equation}
Since $z_a<-1$, the solution is given by $z_a=-1-\sqrt{a}-a/3+\Or(a^{3/2})$ and the rate function satisfies
\begin{equation}
 r^\mathrm{flat}(a)=(2\sqrt{a}+\Or(a^{3/2}))\left(\frac{2}{3}a+\Or(a^{3/2})\right)=\frac{4}{3}a^{3/2}+\Or(a^2).
\end{equation}

To obtain the large $a$ behaviour, notice first that by $-1<\phi(z_a)<0$, the identity \eqref{eqza} gives us
\begin{equation}
 -z_a=1+\frac{a}{\phi(z_a)+1}>1+a.
\end{equation}
As the function $z\mapsto ze^z$ is of order $z$ for small $z$, so is its inverse $L_0(z)$, which implies 
$\phi(z_a)=\Or(z_ae^{z_a})=\Or(ae^{-a})$
and consequently $z_a=-1-a+\Or(a^2e^{-a})$.
Inserting this into \eqref{eqFlRate} completes the proof.
\end{proof}

\section{Poisson initial condition}
We start with a corollary of \cite[Proposition~2.4]{FSW15}:
\begin{proposition}\label{propKernel}
For any $0<\rho<1$ and $n\geq0$, it holds
\begin{equation}\label{eq33}
\Pb\bigg(x^{\textrm{stat},\rho}_n(t)\leq s\bigg)=\bigg(1+\frac{1}{1-\rho}\frac{\D}{\D s}\bigg)\det(\Id-P_s K_{n,t}^\textrm{stat} P_s)_{L^2(\R)},
\end{equation}
where $P_s(\xi)=\Id_{\xi>s}$ and the kernel $K_{n,t}^\textrm{stat}$ is given by
\begin{equation}\label{eqKtStat}
\begin{aligned}	K_{n,t}^\textrm{stat}(\xi_1,\xi_2)&=K_{n,t}^\mathrm{packed}(\xi_1,\xi_2)+(1-\rho)f(\xi_1)g_\rho(\xi_2).
\end{aligned}
\end{equation}
Here, $K_{n,t}^\mathrm{packed}(\xi_1,\xi_2)$ is defined as in~\eqref{eqKpacked}, and:
	\begin{equation}\begin{aligned}
	f(\xi)&=\frac{1}{2\pi\I}\int_{\I\R-\e}\D w\,\frac{e^{{t} (w^2-1)/2+\xi(w+1)}(-w)^{n}}{w+1},\\
	g_\rho(\xi)&=\frac{1}{2\pi\I}\oint_{\Gamma_{0,-\rho}}\D z\,\frac{e^{-{t} (z^2-1)/2-\xi(z+1)}(-z)^{-n}}{z+\rho},
\end{aligned}\end{equation}
for any fixed $0<\e<1$.
\end{proposition}
To directly analyze the stationary case $\rho=1$, the formula \eqref{eq33} has to be analytically continued, which entails some effort:
\begin{proposition}\label{propKernel1}
For any $n\geq0$, it holds
\begin{equation}\label{eq41}
\Pb\bigg(x^\textrm{stat}_n(t)\leq s\bigg)=\frac{\D}{\D s}\bigg( G_n(s)\cdot\det(\Id-P_s K_{n,t}^\mathrm{packed} P_s)_{L^2(\R)}\bigg),
\end{equation}
where $P_s(\xi)=\Id_{\xi>s}$. The function $G_n(s)$ is given by
\begin{equation}%\label{eqKtStat}
\begin{aligned}	
G_n(s)&=s-n-t+R_n(s)- \langle(\Id-P_s K_{n,t}^\mathrm{packed} P_s)^{-1}P_sf^*_s,P_sg_1\rangle,
\end{aligned}
\end{equation}
where
	\begin{equation}\begin{aligned}
	R_n(s)&=\frac{-1}{2\pi\I}\oint_{\Gamma_0}\D z\,e^{-t (z^2-1)/2-s (z+1)}\frac{(-z)^{-n}}{(1+z)^2},\\
	f^*_s(\xi)&=\frac{1}{2\pi\I}\int_{\I\R-1-\e}\D w\,\frac{e^{{t} (w^2-1)/2+\xi(w+1)}(-w)^{n}}{w+1}\\
	&+\frac{1}{(2\pi\I)^2}\int_{\I\R-\e}\D w\,\oint_{\Gamma_0}\D z\,\frac{e^{{t} (w^2-1)/2+\xi(w+1)}}{e^{{t} (z^2-1)/2+s(z+1)}}\frac{(-w)^{n}}{(-z)^{n}}\frac{1}{(w-z)(1+z)},
\end{aligned}\end{equation}
for any fixed $0<\e<1$.
\end{proposition}
\begin{proof}
 We want to take the limit of \eqref{eq33} as $\rho\to1$. The left hand side of \eqref{eq33} is analytic in $\rho$, as the transition density is smooth, and the initial condition is analytic in $\rho$. It remains to obtain the limit of the right hand side. 
 
 Let $\rho=1-\delta$. We factorize the Fredholm determinant as follows:
 \begin{equation}\label{eq32}\begin{aligned}
   \det(\Id-P_s K_{n,t}^\mathrm{stat} P_s)&=\det(\Id-P_s K_{n,t}^\mathrm{packed} P_s)\det(\Id-\delta(\Id-P_s K_{n,t}^\mathrm{packed} P_s)^{-1}P_s f\odot P_s g_\rho)\\
   &=\det(\Id-P_s K_{n,t}^\mathrm{packed} P_s)\left(1-\delta\left\langle(\Id-P_s K_{n,t}^\mathrm{packed} P_s)^{-1}P_s f, P_s g_\rho\right\rangle\right).
 \end{aligned}\end{equation}
The operator identity
\begin{equation}
 (\Id-P_s K_{n,t}^\mathrm{packed} P_s)^{-1}=\Id+(\Id-P_s K_{n,t}^\mathrm{packed} P_s)^{-1}P_s K_{n,t}^\mathrm{packed} P_s
\end{equation}
and the residue calculation
 \begin{equation}\begin{aligned}
  f(\xi)&=1+\frac{1}{2\pi\I}\int_{\I\R-1-\e}\D w\,\frac{e^{{t} (w^2-1)/2+\xi(w+1)}(-w)^{n}}{w+1}=:1+\tilde{f}(\xi)
 \end{aligned}\end{equation}
allow for decomposing the scalar product as a sum.
 \begin{equation}\label{scPr}\begin{aligned}
 \big\langle(\Id-&P_s K_{n,t}^\mathrm{packed} P_s)^{-1}P_s f, P_s g_\rho\big\rangle\\&=\langle P_s1,P_sg_\rho\rangle+\big\langle(\Id-P_s K_{n,t}^\mathrm{packed} P_s)^{-1}P_s(K_{n,t}^\mathrm{packed} P_s1+\tilde{f}), P_s g_\rho\big\rangle
 \end{aligned}\end{equation}
The first summand contains the desired $1/\delta$ term that allows for cancellation of the diverging terms by
\begin{equation}\begin{aligned}
 \langle P_s1,P_sg_\rho\rangle&=\int_s^\infty\D x\,\frac{1}{2\pi\I}\oint_{\Gamma_{0,-\rho}}\D z\,\frac{e^{-{t} (z^2-1)/2-x (z+1)}(-z)^{-n}}{z+\rho}\\
 &=\int_s^\infty\D x\left(\frac{e^{{t} (\delta-\delta^2/2)-x \delta}}{(1-\delta)^{n}}+\frac{1}{2\pi\I}\oint_{\Gamma_{0}}\D z\,\frac{e^{-{t} (z^2-1)/2-x (z+1)}(-z)^{-n}}{z+\rho}\right)\\%\frac{e^{{t} (\delta-\delta^2/2)-s \delta}}{(1-\delta)^{n}\delta}
 &=\frac{1}{\delta}+n+t-s+\Or(\delta)+\frac{1}{2\pi\I}\oint_{\Gamma_{0}}\D z\,\frac{e^{-{t} (z^2-1)/2-s (z+1)}(-z)^{-n}}{(z+1)(z+\rho)}.
 \end{aligned}\end{equation}
Recognizing $K_{n,t}^\mathrm{packed} P_s1+\tilde{f}=f^*_s$ and inserting this into \eqref{scPr}, we obtain
\begin{equation}
 \frac{1}{\delta}-\big\langle(\Id-P_s K_{n,t}^\mathrm{packed} P_s)^{-1}P_s f, P_s g_\rho\big\rangle\to G_n(s),
\end{equation}
as $\delta\to0$, which, together with \eqref{eq32}, completes the proof.
\end{proof}

%\begin[Proof of Theorem~\ref{thmStatRate}]{proof}
\begin{proof}[Proof of Theorem~\ref{thmStatRate}]
We begin by inserting the scaling into \eqref{eq41}.
 \begin{equation}\label{eq53}
  \Pb\bigg(x^\textrm{stat}_t(t)\leq 2t+at\bigg)=\frac{\D}{\D s}\bigg( \widehat{G}_t(s)\cdot\det(\Id-P_s \widehat{K}_t^\mathrm{packed} P_s)_{L^2(\R)}\bigg)\bigg|_{s=0},
 \end{equation}
 where we recall from the packed initial conditions:
\begin{equation}\begin{aligned}	\widehat{K}_{t}^\mathrm{packed}(\xi_1,\xi_2)&=\frac{1}{(2\pi\I)^2}\int_{\gamma_-}\D w\,\oint_{\gamma_+}\D z\,e^{tH(w)-tH(z)}\frac{e^{\xi_1w-\xi_2z}}{w-z},\\
H(w)&=\frac{w^2-1}{2}+(2+a)(w+1)+\log(-w).
\end{aligned}\end{equation}
The function $\widehat{G}_t(s)$ is given by
\begin{equation}%\label{eqKtStat}
\widehat{G}_t(s)=s+at+\widehat{R}_t(s)- \langle(\Id-P_s \widehat{K}_t^\mathrm{packed} P_s)^{-1}P_s\widehat{f}^*_s,P_s\widehat{g}_1\rangle,
\end{equation}
where the individual components are expressed as contour integrals with integration paths already deformed as in Section~\ref{secPac}.
	\begin{equation}\begin{aligned}
	\widehat{R}_t(s)&=\frac{-1}{2\pi\I}\oint_{\gamma_+}\D z\,e^{-tH(z)}\frac{e^{-s(z+1)}}{(1+z)^2}\\
	\widehat{f}^*_s(\xi)&=\frac{1}{2\pi\I}\int_{\gamma_-}\D w\,e^{tH(w)}\frac{e^{\xi(w+1)}}{w+1}\\
	&+\frac{1}{(2\pi\I)^2}\int_{\gamma_-}\D w\,\oint_{\gamma_+}\D z\,e^{{t} (H(w)-H(z))}\frac{e^{\xi(w+1)-s(z+1)}}{(w-z)(1+z)}\\
	\widehat{g}_1(\xi)&=\frac{1}{2\pi\I}\oint_{\Gamma_{0,-1}}\D z\,e^{-tH(z)}\frac{e^{-\xi(z+1)}}{z+1} = 1 + \frac{1}{2\pi\I}\oint_{\gamma_+}\D z\,e^{-tH(z)}\frac{e^{-\xi(z+1)}}{z+1}
\end{aligned}\end{equation}
All integrals are applicable to steepest descent analysis in perfect analogy to the packed initial conditions case. The pointwise limits of the functions are as follows:
\begin{equation}\label{eqRfgAsym}\begin{aligned}
	\widehat{R}_t(s)&=\Or\left(t^{-1/2}e^{-tH(w_+)}\right),\\
	\widehat{f}^*_s(\xi)&=\Or\left(t^{-1/2}e^{tH(w_-)}\right)+\Or\left(t^{-1}e^{t(H(w_-)-H(w_+))}\right),\\
	\widehat{g}_1(\xi)&=1+\Or\left(t^{-1}e^{-tH(w_+)}\right),\\
	\widehat{K}_{t}^\mathrm{packed}(\xi_1,\xi_2)&=\Or\left(t^{-1}e^{t(H(w_-)-H(w_+))}\right),\\
	\det(\Id-P_s \widehat{K}_t^\mathrm{packed} P_s)&=1+\Or\left(t^{-1}e^{t(H(w_-)-H(w_+))}\right).
\end{aligned}\end{equation}
Notice again that $H(w_+)>0$ and $H(w_-)<0$, so all arguments of the exponentials are negative, i.e. the $\mathcal{O}$-terms do in fact \emph{decay} exponentially.

The operator inverse in $\widehat{G}_t(s)$ can be avoided by defining $\widehat{F}_t(s)=s+at+\widehat{R}_t(s)-1$ and recognizing
\begin{equation}\label{eqF}\begin{aligned}
    \widehat{G}_t(s)&\det(\Id-P_s \widehat{K}_t^\mathrm{packed} P_s)_{L^2(\R)} \\
    &= \left(\widehat{F}_t(s) + \det(\Id - (\Id-P_s \widehat{K}_t^\mathrm{packed} P_s)^{-1}P_s\widehat{f}^*_s\odot P_s\widehat{g}_1)\right)\det(\Id-P_s \widehat{K}_t^\mathrm{packed} P_s)\\
    &=\widehat{F}_t(s)\det(\Id-P_s \widehat{K}_t^\mathrm{packed} P_s) +
    \det(\Id-P_s (\widehat{K}_t^\mathrm{packed} + \widehat{f}^*_s\odot \widehat{g}_1) P_s).
\end{aligned}\end{equation}
We divide the derivative of the first summand of \eqref{eqF} by the product rule and investigate each part separately. 
$\widehat{F}_t(s)=s+at-1+\Or(\sqrt{t}e^{-tH(w_+)})$ follows immediately from~\eqref{eqRfgAsym}, and $\widehat{F}'_t(s)=1+\Or(\sqrt{t}e^{-tH(w_+)})$ is easily derived from a straightforward steepest descent analysis of $\widehat{R}'_t(s)$. We can directly compute the derivative of the Fredholm determinant as
\begin{equation}\label{eq62}\begin{aligned}
 \frac{\D}{\D s}\det(\Id-P_s \widehat{K}_t^\mathrm{packed} P_s)&=\sum_{k=1}^\infty\frac{(-1)^k}{k!} \frac{\D}{\D s}\int_{(s,\infty)^k}\D \xi_1\dots\D \xi_k\det_{1\leq i,j\leq k}[\widehat{K}_t^\textrm{packed}(\xi_i,\xi_j)]\\
 &=\sum_{k=1}^\infty\frac{(-1)^{k-1}}{(k-1)!} \int_{(s,\infty)^{k-1}}\D \xi_2\dots\D \xi_k\det_{1\leq i,j\leq k}[\widehat{K}_t^\textrm{packed}(\xi_i,\xi_j)]\bigg|_{\xi_1=s}
\end{aligned}\end{equation}
and again use \eqref{eq12} together with the Hadamard inequality to obtain the bound
\begin{equation}
    \begin{aligned}
        |\eqref{eq62}|\Big|_{s=0}&\leq\sum_{k=1}^\infty\frac{1}{(k-1)!} \int_{\R_+^{k-1}}\D \xi_2\dots\D \xi_k\, k^{k/2}\prod_{i=1}^kC_1e^{-a\xi_i}e^{-tr^\mathrm{packed}_+(a)}\bigg|_{\xi_1=0}\\
        &\leq e^{-tr^\mathrm{packed}_+(a)}\sum_{k=1}^\infty\frac{k^{k/2}}{(k-1)!}\frac{C_1^k}{a^{k-1}} < \infty.
    \end{aligned}
\end{equation}
Collecting terms and inspecting the leading order results in
\begin{equation}\label{eq5.22}
    \frac{\D}{\D s}\left(\widehat{F}_t(s)\det(\Id-P_s \widehat{K}_t^\mathrm{packed} P_s)\right) = 1+\Or(\sqrt{t}e^{-tH(w_+)}).
\end{equation}
For the other term we first need uniform bounds on $\widehat{f}^*_s$, $\frac\D{\D s}\widehat{f}^*_s$ and $\widehat{g}_1$. 
We can estimate the contour integrals by applying steepest descent on the absolute values as has been done in \eqref{eq12}:
\begin{equation}\begin{aligned}
    \left|\frac{1}{2\pi\I}\int_{\gamma_-}\D w\,e^{tH(w)}\frac{e^{\xi(w+1)}}{w+1}\right| &\leq e^{\xi\Re(w+1)}\frac{1}{2\pi}\int_{\gamma_-}|\D w|\,\left|\frac{e^{tH(w)}}{w+1}\right|\leq C_2e^{-\xi a}e^{tH(w_-)}.
\end{aligned}\end{equation}
Similarly,
\begin{equation}\begin{aligned}
\left|\frac{1}{(2\pi\I)^2}\int_{\gamma_-}\D w\,\oint_{\gamma_+}\D z\,e^{{t} (H(w)-H(z))}\frac{e^{\xi(w+1)}}{(w-z)(1+z)}\right|&\leq C_3e^{-\xi a}e^{-tr^\mathrm{packed}}\\
\left|\frac{1}{(2\pi\I)^2}\int_{\gamma_-}\D w\,\oint_{\gamma_+}\D z\,e^{{t} (H(w)-H(z))}\frac{e^{\xi(w+1)}}{(w-z)(1+z)^2}\right|&\leq C_4e^{-\xi a}e^{-tr^\mathrm{packed}}\\
    \left|\frac{1}{2\pi\I}\oint_{\gamma_+}\D z\,e^{-tH(z)}\frac{e^{-\xi(z+1)}}{z+1}\right|&\leq C_5.
\end{aligned}\end{equation}
All bounds hold for $a>0$, $t>1$, $\xi>0$. Taken together, they imply existence of a constant $C_6$ such that the kernel defined by $\mathbf{K} = \widehat{K}_t^\mathrm{packed} + \widehat{f}^*_s\odot \widehat{g}_1$ and its derivative are bounded for $s=0$ as follows:
\begin{equation}\label{eqbfKbound}\begin{aligned}    
    \mathbf{K}(\xi_1,\xi_2) &\leq C_6e^{-\xi_1 a}e^{tH(w_-)},\\
    \frac\D{\D s}\mathbf{K}(\xi_1,\xi_2) &\leq C_6e^{-\xi_1 a}e^{tH(w_-)}.
\end{aligned}
\end{equation}
Now we can bound the derivative of the second summand of \eqref{eqF} similarly to \eqref{eq62}, but recognizing the extra term coming from $\mathbf{K}$ depending explicitly on $s$: 
\begin{equation}\label{eq5.26}\begin{aligned}
 \frac{\D}{\D s}\det(\Id-P_s \mathbf{K} P_s)=&\sum_{k=1}^\infty\frac{(-1)^{k-1}}{(k-1)!} \int_{(s,\infty)^{k-1}}\D \xi_2\dots\D \xi_k\det_{1\leq i,j\leq k}[\mathbf{K}(\xi_i,\xi_j)]\bigg|_{\xi_1=s}\\
 &+\sum_{k=1}^\infty\frac{(-1)^{k}}{(k-1)!} \int_{(s,\infty)^k}\D \xi_1\dots\D \xi_k\det_{1\leq i,j\leq k}[\mathbf{K}^{(?)}(\xi_i,\xi_j)],
\end{aligned}\end{equation}
where, abusing notation, $\mathbf{K}^{(?)}$ is understood as representing $\mathbf{K}$ \emph{or} $\frac\D{\D s}\mathbf{K}$. With both obeying the same bound \eqref{eqbfKbound}, we arrive at
\begin{equation}
    |\eqref{eq5.26}|\Big|_{s=0}\leq e^{tH(w_-)}\sum_{k=1}^\infty\frac{k^{k/2}C_6^k}{(k-1)!}\left(a^{-k+1}+a^{-k}\right) < \infty.
\end{equation}
Since $H(w_+)<-H(w_-)$, this contribution from the second summand decays in fact faster than the one coming from the first summand \eqref{eq5.22}. 
To see this inequality, first notice that
\begin{equation}
 \frac{\D}{\D a} H(w_\pm)=\frac{\partial H}{\partial a}(w_\pm)+H'(w_\pm)\frac{\partial w_\pm}{\partial a}=\frac{\partial H}{\partial a}(w_\pm)=w_\pm+1,
\end{equation}
so we have
\begin{equation}
 \begin{aligned}
  H(w_+)+H(w_-)\Big|_{a=0}&=0\\
  \frac{\D}{\D a}\big(H(w_+)+H(w_-)\big)&=2+w_++w_-=-a<0.
 \end{aligned}
\end{equation}
This concludes the proof, with the rate function being calculated as
\begin{equation}\begin{aligned}
 r^\mathrm{stat}(a)&=\lim_{t\to\infty}t^{-1}\log\Pb\bigg(x_t(t)\geq 2t+at\bigg)
 =H(w_+)=\max_{x\in(-1,0)}H(w)\\
 &=-\frac{a^2}{4}+\left(1+\frac{a}{2}\right)\sqrt{a+\frac{a^2}{4}}+\log\left(1+\frac{a}{2}-\sqrt{a+\frac{a^2}{4}}\right).
\end{aligned}\end{equation}
\end{proof}

%%%%%%%%%%%%%%%%%%%%%%%%%%%%%%%%%%%%%%%%%%%%%%%%%%%%%%%%%%%%%%%%%%%
%%                                                               %%
%% Use the two commands below for producing your bibliography    %%
%% with bibtex, then comment again the commands and include the  %%
%% content of the .bbl file in this file below the commands.     %%
%%                                                               %%
%%%%%%%%%%%%%%%%%%%%%%%%%%%%%%%%%%%%%%%%%%%%%%%%%%%%%%%%%%%%%%%%%%%

\bibliographystyle{alpha}
%\bibliography{Biblio}

% add below the content of your .bbl file produced by bibtex.

\end{document}